\documentclass[12pt]{amsart}

\usepackage{amsmath,amssymb,amsbsy,amsfonts,amsthm,latexsym,
                        amsopn,amstext,amsxtra,euscript,amscd,mathrsfs,color,bm, cite}
                        \usepackage{ulem}

\usepackage{float}
\usepackage[english]{babel}
\usepackage{mathtools}
\usepackage{todonotes}
\usepackage{url}
\usepackage[colorlinks,linkcolor=blue,anchorcolor=blue,citecolor=blue,backref=page]{hyperref}

\usepackage[norefs,nocites]{refcheck}

\newtheorem{theorem}{Theorem}
\newtheorem{lemma}[theorem]{Lemma}

\newtheorem{rem}[theorem]{Remark}

\numberwithin{equation}{section}
\numberwithin{theorem}{section}
\numberwithin{table}{section}

\numberwithin{figure}{section}

\newfont{\teneufm}{eufm10}
\newfont{\seveneufm}{eufm7}
\newfont{\fiveeufm}{eufm5}
\newfam\eufmfam
                \textfont\eufmfam=\teneufm \scriptfont\eufmfam=\seveneufm
                \scriptscriptfont\eufmfam=\fiveeufm
\def\bbbc{{\mathchoice {\setbox0=\hbox{$\displaystyle\rm C$}\hbox{\hbox
to0pt{\kern0.4\wd0\vrule height0.9\ht0\hss}\box0}}
{\setbox0=\hbox{$\textstyle\rm C$}\hbox{\hbox
to0pt{\kern0.4\wd0\vrule height0.9\ht0\hss}\box0}}
{\setbox0=\hbox{$\scriptstyle\rm C$}\hbox{\hbox
to0pt{\kern0.4\wd0\vrule height0.9\ht0\hss}\box0}}
{\setbox0=\hbox{$\scriptscriptstyle\rm C$}\hbox{\hbox
to0pt{\kern0.4\wd0\vrule height0.9\ht0\hss}\box0}}}}
\def\bbbq{{\mathchoice {\setbox0=\hbox{$\displaystyle\rm
Q$}\hbox{\raise 0.15\ht0\hbox to0pt{\kern0.4\wd0\vrule
height0.8\ht0\hss}\box0}} {\setbox0=\hbox{$\textstyle\rm
Q$}\hbox{\raise 0.15\ht0\hbox to0pt{\kern0.4\wd0\vrule
height0.8\ht0\hss}\box0}} {\setbox0=\hbox{$\scriptstyle\rm
Q$}\hbox{\raise 0.15\ht0\hbox to0pt{\kern0.4\wd0\vrule
height0.7\ht0\hss}\box0}} {\setbox0=\hbox{$\scriptscriptstyle\rm
Q$}\hbox{\raise 0.15\ht0\hbox to0pt{\kern0.4\wd0\vrule
height0.7\ht0\hss}\box0}}}}
\def\bbbt{{\mathchoice {\setbox0=\hbox{$\displaystyle\rm
T$}\hbox{\hbox to0pt{\kern0.3\wd0\vrule height0.9\ht0\hss}\box0}}
{\setbox0=\hbox{$\textstyle\rm T$}\hbox{\hbox
to0pt{\kern0.3\wd0\vrule height0.9\ht0\hss}\box0}}
{\setbox0=\hbox{$\scriptstyle\rm T$}\hbox{\hbox
to0pt{\kern0.3\wd0\vrule height0.9\ht0\hss}\box0}}
{\setbox0=\hbox{$\scriptscriptstyle\rm T$}\hbox{\hbox
to0pt{\kern0.3\wd0\vrule height0.9\ht0\hss}\box0}}}}
\def\bbbs{{\mathchoice
{\setbox0=\hbox{$\displaystyle     \rm S$}\hbox{\raise0.5\ht0\hbox
to0pt{\kern0.35\wd0\vrule height0.45\ht0\hss}\hbox
to0pt{\kern0.55\wd0\vrule height0.5\ht0\hss}\box0}}
{\setbox0=\hbox{$\textstyle        \rm S$}\hbox{\raise0.5\ht0\hbox
to0pt{\kern0.35\wd0\vrule height0.45\ht0\hss}\hbox
to0pt{\kern0.55\wd0\vrule height0.5\ht0\hss}\box0}}
{\setbox0=\hbox{$\scriptstyle      \rm S$}\hbox{\raise0.5\ht0\hbox
to0pt{\kern0.35\wd0\vrule height0.45\ht0\hss}\raise0.05\ht0\hbox
to0pt{\kern0.5\wd0\vrule height0.45\ht0\hss}\box0}}
{\setbox0=\hbox{$\scriptscriptstyle\rm S$}\hbox{\raise0.5\ht0\hbox
to0pt{\kern0.4\wd0\vrule height0.45\ht0\hss}\raise0.05\ht0\hbox
to0pt{\kern0.55\wd0\vrule height0.45\ht0\hss}\box0}}}}
\def\bbbz{{\mathchoice {\hbox{$\sf\textstyle Z\kern-0.4em Z$}}
{\hbox{$\sf\textstyle Z\kern-0.4em Z$}} {\hbox{$\sf\scriptstyle
Z\kern-0.3em Z$}} {\hbox{$\sf\scriptscriptstyle Z\kern-0.2em
Z$}}}}

\def\squareforqed{\hbox{\rlap{$\sqcap$}$\sqcup$}}
\def\qed{\ifmmode\squareforqed\else{\unskip\nobreak\hfil
\penalty50\hskip1em\null\nobreak\hfil\squareforqed
\parfillskip=0pt\finalhyphendemerits=0\endgraf}\fi}

\def\cD{{\mathcal D}}

\def\cG{{\mathcal G}}
\def\cH{{\mathcal H}}
\def\cI{{\mathcal I}}
\def\cJ{{\mathcal J}}

\def\cM{{\mathcal M}}

\def\cR{{\mathcal R}}

\def\cW{{\mathcal W}}
\def\cX{{\mathcal X}}

\def\le{\leqslant}

\def\ge{\geqslant}

\newcommand{\ignore}[1]{}

\def\e{\mathbf{e}}

\def \F{\mathbb{F}}

\def\mand{\qquad\mbox{and}\qquad}

\def\\{\cr}
\def\({\left(}
\def\){\right)}
\def\fl#1{\left\lfloor#1\right\rfloor}

\def\e{{\mathbf{\,e}}}
\def\ep{{\mathbf{\,e}}_p}

\def\dsum{\mathop{\quad \sum \ \sum}}

\begin{document}

\title[Binary and ternary congruences]{Binary and ternary congruences involving intervals and sets modulo a prime}

\author[M.~Z.~Garaev]{Moubariz~Z.~Garaev}
\address{Centro  de Ciencias Matem{\'a}ticas,  Universidad Nacional Aut\'onoma de
M{\'e}\-xico, C.P. 58089, Morelia, Michoac{\'a}n, M{\'e}xico}
\email{garaev@matmor.unam.mx}

\author[J. C. Pardo] {Julio C. Pardo}
\address{Centro  de Ciencias Matem{\'a}ticas,  Universidad Nacional Aut\'onoma de
M{\'e}\-xico, C.P. 58089, Morelia, Michoac{\'a}n, M{\'e}xico}
\email{jcpardo@matmor.unam.mx}

\author[I. E. Shparlinski] {Igor E. Shparlinski} \thanks{I.S. was  partially supported by ARC Grants
DP230100530 and DP230100534.}
\address{School of Mathematics and Statistics, University of New South Wales, Sydney, NSW 2052, Australia}
\email{igor.shparlinski@unsw.edu.au}

\begin{abstract} 
Let $s$ be a fixed positive integer constant, $\varepsilon$ be a fixed small positive number.
Then, provided that  a prime $p$ is large enough, we prove that for any set  $\cM\subseteq \F_p^*$ 
of size $|\cM|= \fl{ p^{14/29}}$ and integer $H=\fl{p^{14/29+\varepsilon}}$, any integer $\lambda$ can be represented in the form
$$
\frac{m_1}{x_1^s}+\frac{m_2}{x_2^s}+\frac{m_3}{x_3^s}\equiv \lambda \bmod p,
$$
with
$$
 m_i\in \cM, \quad 1\le x_i\le H, \qquad i=1,2,3.
$$

When $s=1$ we  show that for almost all primes $p$ the following holds:
if $|\cM|= \lfloor p^{1/2}\rfloor$ and $H=\lfloor p^{1/2}(\log p)^{6+\varepsilon}\rfloor$, then any integer 
$\lambda$ can be represented in the form
$$
\frac{m_1}{x_1}+\frac{m_2}{x_2}\equiv \lambda \bmod p,
$$
with
$$
 m_i\in \cM, \quad 1\le x_i\le H, \qquad i=1,2.
$$
\end{abstract}  


\maketitle

\section{Introduction}

\subsection{Motivation} 
The problem of distribution of elements of sets involving products and sums of intervals and arbitrary sets in finite fields has been a subject of 
extensive investigation,
see, for example~\cite{BagSh,BaSh, Gar1, GG1, GK1, GSh1,Shp1} and the references therein. 
Here we continue studying this problem and consider some question related to binary and ternary congruences modulo  a sufficiently large prime number $p$, which involve products and inverses of  elements from short intervals or arbitrary sets. 
 
Let $\F_p$ denote the finite field of residue classes modulo $p$.  Throughout the paper, we always assume that $\F_p$ is represented
by the set $\{0, 1, \ldots, p-1\}$ and we freely alternate between equations
in $\F_p$ and congruences modulo $p$. 
We also use  $\varepsilon>0$ to be an arbitrarily small real number.

Let $\cM$ be an arbitrary subset of $\F_p^*:=\F_p\setminus\{0\}$ with $M=\#\cM$ elements,  $\cX\subseteq \F_p^*$ be an arbitrary interval of length $H$ and $s$ be a  fixed positive integer. It is known from~\cite{GSh1} that if  
$$
H>  p^{17/35+\varepsilon}, \quad M > p^{17/35+\varepsilon}.
$$  
then the number $T_6(\lambda)$ of solutions of the congruence
\begin{align*}
\frac{m_1}{x_1^s}+ \frac{m_2}{x_2^s}+ \frac{m_3}{x_3^s}&+\frac{m_4}{x_4^s}+ \frac{m_5}{x_5^s}+\frac{m_6}{x_6^s}
\equiv \lambda\mod  p, \\
  m_i\in \cM, & \ x_i  \in \cX, \quad i =1, \ldots, 6,
\end{align*}
satisfies
$$
T_6(\lambda)=\frac{H^6 M^6}{p}\(1+O\(p^{-\delta}\)\),
$$
where  $\delta> 0$ depends only on $s$ and $\varepsilon$. It has been posed an open question about obtaining an asymptotic formula for the $5$-term analogues of $T_6(\lambda)$ for  $H= M = o\(p^{1/2}\)$. 
 
\subsection{Main results} 
The first result of our present paper deals with the solubility problem of the corresponding congruence with three summands, assuming that the set $\cX$ starts from the origin. Thus, although our result does not provide asymptotic formula, we can still prove existence of solutions. 

In what follows we use $\cH$ to denote the set of the first $H$ positive integers:
$$
\cH=\{1,2,\ldots, H\}.
$$
We also recall that $\cM$ is an arbitrary subset of $\F_p^*$ with $M=\#\cM$ elements.

\begin{theorem}
\label{thm:three summands}
Let  $M= \lfloor p^{14/29}\rfloor$ and 
$H=\lfloor p^{14/29+\varepsilon}\rfloor $. Then any integer $\lambda$ can be represented in the form
$$
\frac{m_1}{x_1^s}+\frac{m_2}{x_2^s}+\frac{m_3}{x_3^s}\equiv \lambda \bmod p,
\qquad m_i\in \cM, \ x_i\in\cH, \quad i=1,2,3.
$$
\end{theorem}

We note that the problem of obtaining an asymptotic formula for the number of representations
of $\lambda$ in the form given by Theorem~\ref{thm:three summands} is still an open problem.

Let us now consider the problem of solubility of the corresponding congruence with two summands, with $M$ and $H$ of the critical size $p^{1/2+o(1)}$. The case $s\ge 2$ turns to be very hard and currently we can not say much about it. Hence, here we  restrict ourselves to the case $s=1$. 

If $\cM$ is a set of $M$ consecutive integers, then the question is reduced to the problem of ratio of two intervals. This problem was considered in~\cite{Gar1, GG1, GK1}. In this case it is known from the proof of~\cite[Theorem~1.7]{GK1} and~\cite[Lemma~5]{GG1} that if $M>cp^{1/2}, H>cp^{1/2}$, then for some suitable positive constant $c$, any integer $\lambda$ 
can be represented in the form
$$
\frac{m_1}{x_1}+\frac{m_2}{x_2}\equiv \lambda \bmod p,
$$
with $m_1, m_2\in \cM$ and $x_1, x_2\in\cH$. However, the method of the aforementioned works dooes not apply for arbitrary sets $\cM$. 

Our next result deals with binary congruence for arbitrary sets $\cM\subseteq \F_p^*$
for  almost all prime modulus $p$.  As usual, by $\pi(X)$  we denote the number of primes not exceeding $X$.

\begin{theorem}
\label{thm:two summands}
Let $0<\varepsilon <0.1$, $M= \lfloor p^{1/2}\rfloor$ and $H=\lfloor p^{1/2}(\log p)^{6+\varepsilon}\rfloor$. Then for all but $\pi(X)/X^{0.01\varepsilon}$ primes $p\le X$,  any integer 
$\lambda$ can be represented in the form
$$
\frac{m_1}{x_1}+\frac{m_2}{x_2}\equiv \lambda \bmod p,\qquad
 m_i\in \cM, \ x_i\in\cH, \quad i=1,2.
$$
\end{theorem}

Let us now consider the  binary problem with ratios of two intervals. First of all, we observe that if $N<p^{1/2}$, then all irreducible fractions
$$
\frac{m}{x},\quad 1\le m, x \le N, 
$$
are distinct modulo $p$. Furthermore, the number of such fractions is asymptotically 
equal to $6N^2/\pi^2$, for the best known error term see~\cite[Chapter~V, Section~5, Equation~(12)]{Walf}.
Hence, if $N=\lfloor cp^{1/2}\rfloor$, where
$c>\pi/\sqrt{12}$, then 
$$
\#\left \{\frac{m}{x}\bmod p:~ 1\le m, \ x \le N\right \}>p/2
$$
and by the pigeon-hole principle we   get the representability of any integer $\lambda$ in the form
$$
\frac{m_1}{x_1}+\frac{m_2}{x_2}\equiv \lambda \bmod p,
$$
with 
$$
1\le m_1,m_2,x_1,x_2\le \lfloor\sqrt{(\pi^2/12 +\varepsilon)p}\rfloor. 
$$ 

Note that, generally speaking, the sizes of the intervals can not be taken $o(p^{1/2})$.
Indeed, it is easy to see that the congruence
$$
\frac{m_1}{x_1}+\frac{m_2}{x_2}\equiv -1\bmod p
$$
has no solutions in positive integers $m_1, m_2, x_1,x_2$ below $\sqrt{p/3}$. Moreover, the congruence
$$
\frac{m_1}{x_1}+\frac{m_2}{x_2}\equiv 0\bmod p
$$
has no solutions in positive integers $m_1, m_2, x_1,x_2$ below $\sqrt{p/2}$.

We now consider the case  of two arbitrary intervals. More precisely,  let $\cI$ and $\cJ$ be  subsets of $\F_p^*$ consisting of $N$ consecutive integers. Then, it follows   from the aforementioned works~\cite{GG1, GK1} that the representability holds when $N>cp^{1/2}$ for some suitable constant $c$.

Therefore, it is interesting to find the best possible constant $c$ in the problem of binary congruences with ratio of intervals. Here, following the argument of~\cite{GG1, GK1}, we show that for arbitrary intervals $\cI$ and $\cJ$ of length
$\lfloor cp^{1/2}\rfloor$ one can take $c=\sqrt{8}$.

\begin{theorem}
\label{thm:ratio of two intervals}
Let $\cI$ and $\cJ$ be subsets of $\F_p^*$ consisting on $\lfloor\sqrt{8p}\rfloor$ consecutive integers modulo $p$. Then any integer $\lambda$ can be represented in the form
$$
\frac{m_1}{x_1}+\frac{m_2}{x_2}\equiv \lambda \bmod p,\quad m_1, m_2\in \cI,\quad x_1, x_2\in\cJ.
$$
\end{theorem}

\subsection{Notation and conventions}
We recall
that the notations $U = O(V)$,  $U \ll V$ and  $V \gg U$  are
all equivalent to the statement that $|U| \le c V$ holds
with some constant $c> 0$. 

Any implied constants in symbols $O$, $\ll$
and $\gg$ may occasionally, where obvious, depend on the
integer parameters $\ell$ and $s,$ the real parameter $\varepsilon >0$ and the fixed constant $c$, and are absolute otherwise.
 
Finally, given two quantities $U$ and $V>0$, we write 
 $U= o(V)$  and $U = V^{o(1)}$ to mean that $|U| \le \psi(V)V$ 
 and  $|U| \le V^{\psi(V)}$, respectively, for some function $\psi$ 
 such that $\psi(V) \to 0$ as $V \to \infty$.

We also use $U\lesssim  V$ to mean that $|U|< V p^{o(1)}$.

Given a set $\cX$, we use $|\cX|$ to denote its cardinality.

For a positive integer $N$ as usual $\pi(N)$ denotes the number  of primes $p \le N$, while for a 
real $z$ and a prime $p$ we use the notation $\ep(z) = \exp(2\pi i z/p)$.

\section{Auxiliary lemmas}

We recall that $s$ is a fixed positive integer constant, $\cM\subseteq \F_p^*$ with
$M=|\cM|<p^{1/2}$ and
$$
\cH =\{1,2,\ldots, H\}\bmod p.
$$

\begin{lemma}
\label{lem:product small interval}
Assume that 
$$
M=|\cM| < p^{1/2}, \qquad  H = |\cH|<p^{1/2}.
$$
Then for the cardinality of the set
$$
\cM\cdot\cH^{-s} =\{m h^{-s} \bmod p:~m\in\cM,\ h\in\cH\},
$$ 
we have the bound
$$
|\cM\cdot\cH^{-s}| \gg \min \{M H/\log p, \, H^2/\log^2p\}.
$$
\end{lemma}

\begin{proof}
Let $J$ be the number of solutions of the congruence
$$
m_1q_1^{-s}\equiv m_2q_2^{-s}\bmod p, \quad m_1, m_2\in\cM,\quad q_1,q_2\in \cR,
$$
where $\cR$ is the set of primes from $\cH$. We have
\begin{equation}
\label{eq: R}
R=|\cR|\gg \frac{H}{\log p.} 
\end{equation}
We claim that,
$$
J\le s(M^2 + MR).
$$
If $m_1=m_2$ then $q_1^s\equiv q_2^s\bmod p$ and we have at most $sM R$ possibilities in this case.

Let $m_1\not=m_2$ be fixed. Denoting $\lambda=m_1m_2^{-1}$, we have
$$
(q_1q_2^{-1})^s\equiv \lambda\bmod p,\quad q_1 \ne q_2.
$$ 
It follows that, for any fixed $\lambda$, there are at most $s$ different values for $q_1q_2^{-1}$ modulo $p$. But for each such value of $q_1 q_2^{-1}$ the primes $q_1$ and $q_2$ are defined uniquely.
Indeed, if there are two such pairs of primes
$q_1 \ne q_2$ and $r_1 \ne r_2$,  then
$$
q_1 r_2\equiv q_2 r_1\bmod p. 
$$
Since both  sides are positive integers less than $p$, this congruence becomes an equality $q_1 r_2 = q_2 r_1$, implying
$q_1=r_1$ and  $q_2=r_2$.

Now we use the  standard relation between the number of solutions and the cardinalities, which follows from the Cauchy-Schwarz 
inequality:
$$
\sum_{m \in \cM} \sum_{q \in \cR}  1  
= \sum_{\lambda\in \cM\cdot\cR^{-s}}
 \dsum_{\substack{m \in \cM, \ q \in \cR\\
mq^{-s}\equiv \lambda \bmod p}}  1  \le  \sqrt {|\cM\cdot\cR^{-s}| J}, 
$$
where $\cR^{-s} = \{q^{-s}:~q \in \cR\}$.  We now derive that
$$
|\cM\cdot\cH^{-s}|\ge |\cM\cdot\cR^{-s}|\ge \frac{s M^2R^2}{s(M^2+MR)}\gg
\min\{R^2, MR\}, 
$$
which together with~\eqref{eq: R} concludes the proof. 
\end{proof}

We also need the following result given by~\cite[Theorem~1.2]{GSh1}.

\begin{lemma}
\label{lem: KloostFrac}  Let  $s$ be a fixed positive integer and let integers $K$ and $H$ be such that
$$
0<K+1<K+H<p.
$$
Then for any fixed positive integer $\ell$ the following bound holds
$$
\sum_{m\in\cM}\left|\sum_{K+1\le  x\le K+H}\e_p(amx^{-s})\right|\lesssim  HM\(\frac{p}{MH^{2\ell/(\ell+1)}}+\frac{1}{M}\)^{1/(2\ell)}.
$$
\end{lemma}

Finally,  we recall the following statement from~\cite[Theorem~3]{Gar1}; 
we refer to~\cite[Chapter~3]{IwKow} or~\cite[Section~4.2]{MonVau} 
for a background on characters.

\begin{lemma}
\label{lem:charestsumprime}
Let  $c$ be a positive constant and let 
$$
Q= (\log X)^c,  \qquad   \Delta > 0, \qquad  \Delta  \to 0 \quad \text{as}\  X\to +\infty.
$$ 
Then  for all, but
$$
O\(X^{2c^{-1}-4(\log\Delta)(\log
Q)^{-1}}e^{-\frac{\log X}{100c\log\log X}}\)
$$
positive integers $k \le X$, the inequality
$$
\left|\sum_{\substack{q\le Q\\q~\text{prime}}}\chi(q)\right|\le \pi(Q)\Delta
$$
holds for any  primitive multiplicative character $\chi$ modulo $k$. 
\end{lemma}
We note that for a prime modulus $k=p$ all non-trivial characters are primitive.

\section{Proof of the theorems}

\subsection{Proof of Theorem~\ref{thm:three summands}}
Let 
$$
\cD= \cM\cdot\cH^{-s}.
$$
By Lemma~\ref{lem:product small interval}, we have
$$
|\cD|\gg \min \{M H/\log p, \, H^2/\log^2p\} = MH/\log p.
$$
It suffices to prove the solubility of the congruence
\begin{equation}
\label{eqn:d1+d2+frac m/xs}
d_1+d_2 + \frac{m}{x^s}\equiv \lambda \bmod p,
\end{equation}
with
$$
d_1, d_2\in \cD,\quad  m\in \cM, \quad x\in\cH.
$$
Using the orthogonality of exponential functions and expressing the number $T$ of solutions of~\eqref{eqn:d1+d2+frac m/xs} 
via exponential sums, we see that 
$$
T =\sum_{d_1\in\cD}\sum_{d_2\in \cD}\sum_{m\in\cM}\sum_{x \in\cH}\frac{1}{p}\sum_{a=0}^{p-1} \e_p\(a(d_1+d_2+mx^{-s}-\lambda)\). 
$$  
Next, changing the order of summation and separating the term corresponding to $a=0$,  we obtain 
\begin{equation}
\label{eqn:T=Main + Error}
\begin{split}
T& =\frac{1}{p}\sum_{a=0}^{p-1}\sum_{d_1\in\cD}\sum_{d_2\in \cD}\sum_{m\in\cM}\sum_{x\in\cH}\e_p\(a(d_1+d_2+mx^{-s}-\lambda)\)\\
& =\frac{|D|^2HM}{p} +O(E),
\end{split} 
\end{equation}
where
\begin{align*}
E& = \frac{1}{p}\left|\sum_{a=1}^{p-1}\sum_{d_1\in\cD}\sum_{d_2\in \cD}\sum_{m\in\cM}\sum_{x\in\cH}\e_p(a(d_1+d_2+mx^{-s}-\lambda))\right| \\ 
& \le  \frac{1}{p}
\sum_{a=1}^{p-1}\left|\sum_{d\in\cD}\e_p(ad)\right|^2\left|\sum_{m\in\cM}\sum_{x\in\cH}\e_p(amx^{-s})\right|.
\end{align*}

From Lemma~\ref{lem: KloostFrac} with $\ell=3$ we have that
$$
\left|\sum_{m\in\cM}\sum_{x\in\cH}\e_p(amx^{-s})\right|\lesssim  HM \(\frac{p}{MH^{3/2}}+\frac{1}{M}\)^{1/6}.
$$
Observing that
$$
\frac{p}{MH^{3/2}}+\frac{1}{M}\ll \frac{p}{MH^{3/2} }
\mand \sum_{a=0}^{p-1}\left|\sum_{d\in\cD}\e_p(ad)\right|^2 = p |D|, 
$$
we get
\begin{align*}
E&\lesssim  HM\frac{p^{1/6}}{p M^{1/6}H^{1/4}}\sum_{a=0}^{p-1}\left|\sum_{d\in\cD}\e_p(ad)\right|^2 \\
& = HM\frac{p^{1/6}}{ M^{1/6}H^{1/4}} |D|=\frac{|D|^2HM}{p}\times\Delta, 
\end{align*}
where
$$
\Delta =\frac{p^{7/6}}{|D|M^{1/6}H^{1/4}}<p^{-0.1\varepsilon}.
$$
Inserting this  into~\eqref{eqn:T=Main + Error}, we obtain that
$$
T=\(1+ o(1)\)\frac{|D|^2HM}{p},
$$
and the result follows.

\subsection{Proof of Theorem~\ref{thm:two summands}}

In Lemma~\ref{lem:charestsumprime} we take 
$$
c=6+\varepsilon,\quad  Q=(\log X)^{6+\varepsilon},\quad \Delta=(\log X)^{-1-0.1\varepsilon}.
$$  
Then we have at most 
$O(X^{1-0.6\varepsilon/(6+\varepsilon)})$ exceptions for $k\le X$. Hence,
for $\pi(X)\(1+O(X^{-0.05\varepsilon})\)$ primes $p$ with $X^{1/2}< p\le  X$ the inequality
\begin{equation}
\label{eqn:char sums our case}
\left|\sum_{q\le Q}\chi(q)\right|\le \pi(Q)(\log X)^{-1-0.1\varepsilon}
\end{equation}
holds for any non-principal character $\chi$ modulo $p$.

Let 
$$
\cD = \left \{m/h:~m\in \cM,\ h\in \left\{1, \ldots, \fl{0.01p^{1/2}}\right\}\right\}.
$$
Note that if $X^{1/2}<p\le X$,  $h\le [0.01p^{1/2}]$ and $q\le (\log X)^{6+\varepsilon}$, then $hq\le p^{1/2}(\log p)^{6+\varepsilon}$. 

We now prove that if~\eqref{eqn:char sums our case} holds for some $X^{1/2}<p\le X$, then for the cardinality of the set 
$$
\cG  =\{d/q\bmod p :~d\in \cD, \  q\le (\log X)^{6+\varepsilon}\}
$$
we have
\begin{equation}
\label{eqn: G large}
|\cG| = \(1+o(1)\)p.
\end{equation}
This   implies that 
$$
\cG+\cG =\F_p,
$$
which together with $\cG\subseteq \cM/\cH\bmod p$, yields  the desired result.

By Lemma~\ref{lem:product small interval} we have that
\begin{equation}
\label{eqn:lower bound cD}
|\cD|\gg \frac{p}{\log^2 p}.
\end{equation}
Let $\cG^{c}=\F_p^*\setminus \cG$.   Thus~\eqref{eqn: G large} is equivalent to $|\cG^{c}|=o(p)$.

By the definition of $\cG$,  there is no solutions to the congruence
$$
d\equiv q y \bmod p,\qquad d\in\cD, \ q\le Q,\ y\in \cG^c. 
$$ 
Hence, using the orthogonality of  multiplicative characters,  
expressing the number of solutions (which is zero) via sums of characters $\chi$ modulo $p$,  we get that
$$
0= \sum_{d\in\cD}\sum_{q\le Q}\sum_{y\in \cG^c} \frac{1}{p-1}\sum_{\chi \bmod p} \chi\(d^{-1}qy\).
$$
Changing the order of summation and separating the term that corresponds to the principal character $\chi=\chi_0$,
we derive
$$
\pi(Q)\cdot |\cD|\cdot |\cG^c|\le 
\sum_{\chi\not=\chi_0}\left|\sum_{q\le Q}\chi(q)\right|\left|\sum_{y\in \cG^c}\chi(y)\right|\left|\sum_{d\in \cD}\chi(d^{-1})\right|.
$$
Using~\eqref{lem:charestsumprime} to estimate the sum over $q\le Q$, and then following the standard procedure with application of the Cauchy-Schwarz inequality, we get that
\begin{equation}
\label{eqn:principal char separate}
\pi(Q)\cdot |\cD|\cdot |\cG^c|\le \pi(Q)(\log X)^{-1-0.1\varepsilon}T_1^{1/2}T_2^{1/2}, 
\end{equation}
where
$$
T_1=\sum_{\chi \bmod p} \left|\sum_{y\in \cG^c}\chi(y)\right|^2=(p-1)|\cG^c|
$$
and
$$
T_2= \sum_{\chi \bmod p} \left|\sum_{d\in \cD}\chi(d^{-1})\right|^2=(p-1)|\cD|.
$$ 
Therefore, from~\eqref{eqn:principal char separate}, we get that
$$
|\cD||\cG^c|\le (p-1)(\log X)^{-1-0.1\varepsilon}|\cD|^{1/2}|\cG^c|^{1/2}.
$$
Thus,
$$
|\cG^c|\le \frac{p^2}{|\cD|(\log X)^{2+0.2\varepsilon}}.
$$
This, together with~\eqref{eqn:lower bound cD}, implies that
$$
|\cG^c|\ll \frac{p}{(\log X)^{0.2\varepsilon}}=o(p),
$$
thus, establishing~\eqref{eqn: G large} and 
finishing the proof of our Theorem~\ref{thm:two summands}.

\begin{rem} It is well known, that under the Generalised Riemann Hypothesis (GRH)
a version of  Lemma~\ref{lem:charestsumprime} holds for all moduli, 
see~\cite[Equation~(13.21)]{MonVau}. Thus under the GRH we have a version 
of Theorem~\ref{thm:two summands} which holds for all sufficiently large $p$. 
\end{rem} 

\subsection{Proof of Theorem~\ref{thm:ratio of two intervals}}
Let 
$$
\cI=\left\{L+1,\ldots, L+\lfloor\sqrt{8p}\rfloor\right\} \quad \text{and}\quad  \cJ=\left\{K+1,\ldots, K+\lfloor\sqrt{8p}\rfloor\right\}.
$$
Denote
$$
\cW= \{m/x\bmod p:~m\in\cI,\ x\in\cJ\},
$$
and let $\cW^c=\F_p\setminus \cW$.
It suffices to prove that $|\cW^c|<p/2$. Indeed,  in this case $|\cW|>p/2$
and the claim then follows  from the pigeon-hole principle.

Let $N =\lfloor\sqrt{2p}\rfloor$ and note that $2N\le \lfloor\sqrt{8p}\rfloor$. Clearly, for $t\in\cW^c$ there is no solutions to the congruence
$$
m_1+m_2\equiv (x_1+x_2)t\bmod p, 
$$
with  positive integers $m_1,m_2,x_1,x_2$ satisfying
$$
m_1\in [L+1,  L+N], \quad x_1\in [K+1, K+N],\quad m_2\le N,\quad x_2\le N.
$$
As in the prove of  Theorem~\ref{thm:three summands},
 we write the number of solutions of this congruence (which is zero) in terms of exponential  sums and see that
$$
0=\frac{1}{p}\sum_{a=0}^{p-1}\sum_{m_1=L+1}^{L+N} \sum_{m_2=1}^N\sum_{x_1=K+1}^{K+N} \sum_{x_2=1}^N\sum_{t\in\cW^c}
\e_p\(a(m_1+m_2 -(x_1+x_2)t)\).
$$
Separating the term corresponding to $a=0$ and following the standard procedure, we derive 
\begin{equation}
\label{eqn:ratio intervals a=0}
N^4|\cW^c|\le \sum_{a=1}^{p-1}S_1(a)S_2(a)\(\sum_{t\in \cW^c}T_1(at)T_2(at)\), 
\end{equation}
where
$$
S_1(a)=\left|\sum_{m_1=L+1}^{L+N} \e_p(am_1)\right|, \qquad S_2=\left|\sum_{m_2=1}^Ne_p(am_2\right|,
$$
and
$$
T_1(at)=\left|\sum_{x_1=K+1}^{K+N} \e_p(-atx_1)\right|,\qquad T_2(at)=\left|\sum_{x_2=1}^N\e_p(-atx_2)\right|.
$$
Now we observe that
\[
\sum_{t\in \cW^c}T_1(at)T_2(at)\le \sum_{t=0}^{p-1}T_1(at)T_2(at)=\sum_{t=0}^{p-1}T_1(t)T_2(t).
\]
Therefore, applying the Cauchy-Schwarz inequality, we  obtain
\[
\sum_{t\in \cW^c}T_1(at)T_2(at) \le  \(\sum_{t=0}^{p-1}T_1^2(t)\)^{1/2}\(\sum_{t=0}^{p-1}T_2^2(t)\)^{1/2}=pN.
\]  
Analogously, we have
$$
\sum_{a=1}^{p-1}S_1(a)S_2(a)= \sum_{a=0}^{p-1}S_1(a)S_2(a) - N^2\le pN-N^2.
$$
Hence, inserting these estimates into~\eqref{eqn:ratio intervals a=0}, we derive
$$
|\cW^c|\le \frac{(p-N)p}{N^2}=\frac{p}{2}\cdot\frac{2p-2N}{N^2}.
$$
Since $N^2+2N = (N+1)^2- 1 > 2p -1$ and $N^2+2N\not = 2p$ 
we see  that 
$$
N^2>2p-2N,
$$
implying $|\cW^c|<p/2$.

\end{document}